\documentclass[reqno]{amsart}
\usepackage{amssymb}
\usepackage{amsfonts}
\usepackage{amsmath}
\usepackage{graphicx}

\newtheorem{theorem}{Theorem}
\theoremstyle{plain}

\newtheorem{claim}{Claim}

\newtheorem{corollary}{Corollary}

\newtheorem{definition}{Definition}
\newtheorem{example}{Example}

\newtheorem{lemma}{Lemma}

\newtheorem{proposition}{Proposition}
\newtheorem{remark}{Remark}

\numberwithin{equation}{section}

\begin{document}
\title[Algorithms to test OSC for self-similar set]{Algorithms to test open
set condition for self-similar set related to P.V. numbers}
\author{Hao Li}
\address{Institute of Mathematics, Zhejiang Wanli University, Ningbo,
Zhejiang, 315100, P. R. China}
\email{kevinlee9809@hotmail.com}
\author{Qiuli Guo}
\address{Institute of Mathematics, Zhejiang Wanli University, Ningbo,
Zhejiang, 315100, P. R. China}
\email{guoqiuli@zwu.edu.cn}
\author{Qin Wang$^\ast$}
\address{School of Computer Science and Information Technology, Zhejiang
Wanli University, Ningbo, Zhejiang, 315100, P. R. China}
\email{qinwang@126.com}
\author{Li-Feng Xi}
\address{Institute of Mathematics, Zhejiang Wanli University, Ningbo,
Zhejiang, 315100, P.~R. China}
\email{xilifengningbo@yahoo.com}
\thanks{$\ast $ Corresponding author. This work is supported by National
Natural Science of Foundation of China (Nos. 11371329, 11071224), NCET and
NSF of Zhejiang Province (Nos. LR13A1010001, LY12F02011, Q14A010014)}

\begin{abstract}
Fix a P.V. number $\lambda ^{-1}>1.$ Given $\mathbf{p}=(p_{1},\cdots
,p_{m})\in \mathbb{N}^{m}$, $\mathbf{b}=(b_{1},\cdots ,b_{m})\in \mathbb{Q}%
^{m}$, for the self-similar set $E_{\mathbf{p},\mathbf{b}}=\cup
_{i=1}^{m}(\lambda ^{p_{i}}E_{\mathbf{p},\mathbf{b}}+b_{i})$ we find an
efficient algorithm to test whether $E_{\mathbf{p},\mathbf{b}}$ satisfies
the open set condition (strong separation condition) or not.
\end{abstract}

\maketitle


\section{Introduction}

Suppose $K=\cup _{i=1}^{m}f_{i}(K)\subset \mathbb{R}^{d}$ is a self-similar
set where $\{f_{i}\}_{i=1}^{m}$ are contracting similitudes, i.e., there are
$r_{i}\in (0,1)$ such that $|f_{i}(x)-f_{i}(y)|=r_{i}|x-y|\text{ for all }%
x,y\in \mathbb{R}^{d}.$
If $\sum_{i=1}^{m}(r_i)^{s}=1,$ we call $s$ the similarity dimension of $K$ and denote it by $\dim_S K$.
We say that the open set condition (\textbf{OSC}) is
fulfilled for $\{f_{i}\}_{i},$ if there is a non-empty open set $U\,$such
that $\cup _{i=1}^{m}f_{i}(U)\subset U\,$and $f_{i}(U)\cap
f_{j}(U)=\varnothing $ for all $i\neq j.$ The strong separation condition (%
\textbf{SSC}) is satisfied, if $f_{i}(K)\cap f_{j}(K)=\varnothing $ for all $%
i\neq j.$ Let $\Sigma ^{\ast }=\cup _{k=1}^{\infty }\{1,\cdots ,m\}^{k},$ $%
f_{\mathbf{i}}=f_{i_1}\circ f_{i_2}\circ\cdots\circ f_{i_k}$ with $\mathbf{i}%
=i_1i_2\cdots i_k\in\{1,\cdots ,m\}^{k}$.

\begin{remark}
The open set condition was introduced by Moran \cite{Moran} and studied by
Hutchinson \cite{Hu}. Schief \cite{S}, Bandt and Graf \cite{B} showed the
relation between the open set condition and the positive Hausdorff measure.
\end{remark}

Self-similar sets with overlaps have very complicated\emph{\ }structures.
Falconer \cite{Falc088} proved some \textquotedblleft
generic\textquotedblright\ result on Hausdorff dimension of self-similar
sets without the assumption about the open set condition. One useful notion
\textquotedblleft transversality\textquotedblright\ to study self-similar
sets (or measures) with overlaps can be found e.g. in Keane, Smorodinsky and
Solomyak \cite{Keane}, Pollicott and Simon \cite{PS}, Simon and Solomyak
\cite{Simon} and Solomyak \cite{So}. For self-similar $\lambda$-Cantor set
\begin{equation*}
E_{\lambda }=E_{\lambda }/3\cup (E_{\lambda }/3+\lambda /3)\cup (E_{\lambda
}/3+2/3),
\end{equation*}
a conjecture of Furstenberg says that $\dim _{\text{H}}E_{\lambda }=1$ for
any $\lambda $ irrational. Recently, Hochman \cite{Ho} solved the
Furstenberg conjecture, also see \cite{SV}.

Kenyon \cite{K} obtained that the \textbf{OSC} is fulfilled for $E_{\lambda
} $ if and only if $\lambda =p/q\in \mathbb{Q}$ with $p\equiv q\not\equiv 0($%
mod$3).$ Rao and Wen \cite{RW} also discussed the structure of $E_{\lambda }$
with $\lambda \in \mathbb{Q}$ using the key idea \textquotedblleft
graph-directed struture\textquotedblright\ introduced by Mauldin and
Williams \cite{MW}. In particular, Rao and Wen \cite{RW} studied the
self-similar sets with \textquotedblleft complete overlaps\textquotedblright
.

In general, when considering the $\{f_{i}(x)=x/n+b_{i}\}_{i=1}^{N}$ with $%
b_{i}\in \mathbb{Q},n\in \mathbb{N}$ and $n\geq 2,$ as mentioned in \cite{Z}%
, we don't have a general classification of those parameters $%
(n,b_{1},\cdots ,b_{N})$ for which the \textbf{OSC} is satisfied. An
interesting observation is that the weak separation property (\textbf{WSP})
fulfills for $\{f_{i}(x)\}_{i=1}^{N}.$

\begin{definition}
We say that the \textbf{WSP} fulfills iff the identity $\mathbf{id}$ is not
an accumulation point of $\{f_{\mathbf{i}}^{-1}f_{\mathbf{j}}:$ $\mathbf{i}%
\neq \mathbf{j\in }$ $\Sigma ^{\ast }\}$.
\end{definition}

Lau and Ngai \cite{L} introduced the \textbf{WSP} and Zerner \cite{Z}
developed their theory. Feng and Lau \cite{FL} studied the multifractal
formalism for self-similar measures with the \textbf{WSP}. It is clear that $%
\mathbf{SSC}\Longrightarrow \mathbf{OSC}\Longrightarrow \mathbf{WSP}.$

Ngai and Wang \cite{NW} introduced the finite type condition (\textbf{FTC})
and described a scheme for computing the exact Hausdorff dimension of
self-similar set in the absence of the \textbf{OSC}. Nguyen \cite{N} showed
that the \textbf{FTC} implies the \textbf{WSP}. Lau and Ngai \cite{LN2}
introduced a generalized finite type condition (\textbf{GFTC}) which
extended a more restrictive condition in \cite{NW} and proved that the
generalized finite type condition implies the \textbf{WSP}. Lau, Ngai and
Wang \cite{LNW} also extended both \textbf{WSP} and \textbf{FTC} to include
finite iterated function systems (\textbf{IFSs}) of injective $C^1$
conformal contractions on compact subsets of $\mathbb{R}^{d}$ . Then we have

\begin{center}
\textbf{FTC\ }$\Longrightarrow$\textbf{\ GFTC }$\Longrightarrow$ \textbf{WSP}%
.
\end{center}
Das and Edgar \cite{DE} proved that the \textbf{GFTC} implies geometric \textbf{%
WSP} for graphs.

Lau and Ngai \cite{L} revealed the connection between the \textbf{WSP}\ and
\textbf{P.V. number}.

\begin{definition}
A Pisot-Vijayaraghavan number, also called simply a P.V. number, is a real
algebraic integer greater than 1 such that all its Galois conjugates are
less than 1 in absolute value.
\end{definition}

The P.V. numbers were discovered by Thue in 1912 and rediscovered by Hardy
in 1919 within the context of Diophantine approximation. They became widely
known after the publication of Charles Pisot's dissertation in 1938. Some
elementary properties of P.V. number will offer help to us:

\begin{itemize}
\item Every positive integer except 1 is a P.V. number;

\item For example, $\frac{\sqrt{5}+1}{2}$ and $\sqrt{2}+1$ are P.V. numbers;

\item Garsia's theorem \cite{G}: Suppose $\eta $ is a P.V. number and $\eta
_{1},\cdots ,\eta _{l}$ are algebraic conjugates of $\eta \ $with $|\eta
_{i}|<1$ for all $i.$ If $L(x)=a_{0}+a_{1}x+\cdots +a_{k}x^{k}$ with integer
coefficients, then
\begin{equation}
L(\eta )=0 \text{\ or\ } L(\eta )\geq \lbrack
\max_{i}|a_{i}|]^{-l}\prod\nolimits_{i=1}^{l}(1-|\eta _{i}|);  \label{garsia}
\end{equation}

\item Given a P.V. number $\lambda ^{-1}>1,$ an interesting fact is that the
IFS $\{\lambda ^{p_{i}}x+b_{i}\}_{i=1}^{m}$ with $p_{i}\in \mathbb{N}$ and $%
b_{i}\in \mathbb{Q}$ for all $i$, by Theorem 2.9 in \cite{NW}, is of \textbf{%
finite type}, and then \textbf{WSP} is fulfilled \cite{L}.
\end{itemize}

\subsection{Decidability on fractals}

$\ $\

It is well known (e.g. see \cite{Si}) that the halting problem for Turing
machines, Wang's tiling problem, Hilbert's tenth problem and group
isomorphism problem have no algorithms, i.e., these problems are
undecidable. Here a \emph{decidable }problem is a question such that there
is an algorithm or computer program to answer \textquotedblleft yes" or
\textquotedblleft no" to the question for every possible input.

There are some related results as follows.

(1) Based on the undecidability of Post correspondence problem (PCP), Dube
\cite{D1,D2} discussed the undecidability of the problem on the invariant
fractals of iterated function systems. For example, Theorem 5 of \cite{D1}
actually shows that given a self-affine set in the plane, it is \emph{%
undecidable} to test if it satisfies the \textbf{SSC}.

(2) With the \textbf{WSP} of similarities $\{S_{i}\}_{i=1}^{n},$ Lau, Ngai
and Rao \cite{LNR} proved that the following problem is \emph{decidable} to
test if the self-similar measure $\mu =\sum\nolimits_{i=1}^{m}\rho _{i}\mu
\circ S_{i}^{-1}$ is absolutely continuous or not. In fact, they obtained
the criterion based on the transition matrix.

An algorithm is said to be of \emph{polynomial time} if its running time is upper
bounded by a polynomial expression in the size of the input for the
algorithm, i.e., $T(n)$ $=O(n^{k})$ for some constant $k$.

Let $G=(\mathcal{V},\mathcal{E})$ be a directed graph with vertex set $%
\mathcal{V}$ and directed-edge set~$\mathcal{E}$. Recall some polynomial
time algorithms as follows:

(a) Dijkstra's algorithm \cite{D} finds the lowest cost path for a graph
with nonnegative edge path costs, Dijkstra's original algorithm does not use
a min-priority queue and runs in $O(|\mathcal{V}|^{2});$

(b) Kosaraju's algorithm \cite{C}\ is an algorithm with its running time $O(|%
\mathcal{V}|^{2})$ to find the strongly connected components of a directed
graph;

(c) Kruskal's algorithm \cite{Kr}\ finds a minimum spanning tree for a
connected weighted graph, this algorithm requires $O(|\mathcal{E}|\log |%
\mathcal{E}|)$ time$.$

\subsection{Main result}

\

In this paper, fix a P.V. number $\lambda ^{-1}>1,$ we consider the IFS%
\begin{equation*}
S_i(x)=\lambda ^{p_{i}}x+b_{i}\text{ for }i=1,\cdots m,
\end{equation*}%
where $\mathbf{p}=(p_{1},\cdots ,p_{m})\in \mathbb{N}^{m}$, $\mathbf{b}%
=(b_{1},\cdots ,b_{m})\in \mathbb{Q}^{m}.$ Let
\begin{equation}  \label{Epb}
E_{\mathbf{p},\mathbf{b}}=\cup _{i=1}^{m}(\lambda ^{p_{i}}E_{\mathbf{p},%
\mathbf{b}}+b_{i})\subset \mathbb{R}^{1}
\end{equation}%
be the self-similar set w.r.t. the IFS $\{\lambda
^{p_{i}}x+b_{i}\}_{i=1}^{m}.$

Although there is no general classification of parameters $(\mathbf{p},%
\mathbf{b})$ for which the \textbf{OSC} is satisfied, in this paper, we
obtain an algorithm to test whether \textbf{OSC} is satisfied or not.

\begin{theorem}
\label{Th:Main}Fix a P.V. number $\lambda ^{-1}>1.$ Given $\mathbf{p}%
=(p_{1},\cdots ,p_{m})\in \mathbb{N}^{m}$, $\mathbf{b}=(b_{1},\cdots
,b_{m})\in \mathbb{Q}^{m}$, the following problems on the self-similar set $$%
E_{\mathbf{p},\mathbf{b}}=\cup _{i=1}^{m}(\lambda ^{p_{i}}E_{\mathbf{p},%
\mathbf{b}}+b_{i})$$ are decidable$:$ \newline
$(1)$ whether $E_{\mathbf{p},\mathbf{b}}$ satisfies the open set condition$;$
\newline
$(2)$ whether $E_{\mathbf{p},\mathbf{b}}$ satisfies the strong separation
condition.
\end{theorem}

\begin{remark}
\cite{WX} deals with the case that $\lambda^{-1}$ is an integer and $p_1=\cdots=p_m=1.$
\end{remark}

\begin{remark}\label{R:a}
\label{remark:int} Without loss of generality, we assume that
\begin{equation}
(b_{1},\cdots ,b_{m})\in \mathbb{Z}^{m}.  \label{key}
\end{equation}%
In fact, we can select $a\in \mathbb{Z}$ such that $ab_{i}\in \mathbb{Z}$
for all $i,$ suppose $T_{i}(x)=\lambda ^{p_{i}}x+ab_{i}.$ Then the \textbf{OSC}
$($or \textbf{SSC}$)$ holds for $\{T_{i}\}$ if and only if the \textbf{OSC} $%
($or \textbf{SSC}$)$ holds for $\{S_{i}\}.$
\end{remark}

\begin{remark}
Since $n\geq 2$ is a P.V. number, in an algorithmic point of view we answer
the problem on $(n,b_{1},\cdots ,b_{m})$ from \cite{Z}.
\end{remark}

\begin{corollary}\label{weishu}
For $E_{\mathbf{p},\mathbf{b}}$ as above, then $ \dim_H E_{\mathbf{p},\mathbf{b}}=\dim_S E_{\mathbf{p},\mathbf{b}}$
if and only if  \textbf{OSC\ } is fulfilled. Then the problem is decidable whether $\dim_H E_{%
\mathbf{p},\mathbf{b}}=\dim_S E_{\mathbf{p},\mathbf{b}}$.
\end{corollary}

\begin{proof}
Since $E_{\mathbf{p},\mathbf{b}}$ is of finite type, we have $0<\mathcal{H}%
^{\dim_H(E_{\mathbf{p},\mathbf{b}})}(E_{\mathbf{p},\mathbf{b}})<\infty$ by
Theorem 1.2 in \cite{NW}. Schief's theorem \cite{S} implies \textbf{OSC} $%
\Longleftrightarrow$ $0<\mathcal{H}^{\dim_S(E_{\mathbf{p},\mathbf{b}})}(E_{%
\mathbf{p},\mathbf{b}})<\infty$. Then
\begin{center}
\textbf{OSC\ }$\Longleftrightarrow$\ $\dim_H E_{\mathbf{p},\mathbf{b}%
}=\dim_S E_{\mathbf{p},\mathbf{b}}$
\end{center}
and the decidability follows from Theorem \ref{Th:Main}.
\end{proof}

We consider $E_{\mathbf{p},\mathbf{b}}$ with $\dim_SE_{\mathbf{p},\mathbf{b}}=1$. Then, by Schief's
theorem, we have

\begin{center}
\textbf{OSC\ }\ \ $\Longleftrightarrow$\ \ $0<\mathcal{H}^1 (E_{\mathbf{p},\mathbf{b}})<%
\infty$\ \ $\Longleftrightarrow$\ \ int$(E_{\mathbf{p},\mathbf{b}})\neq \varnothing$,
\end{center}
where int$(\cdot)$ denotes interior of set.

\begin{corollary}\label{int}
For $E_{\mathbf{p},\mathbf{b}}$ with $\dim_SE_{\mathbf{p},\mathbf{b}}=1$, then \emph{int}$(E_{\mathbf{p},\mathbf{b}})\neq \varnothing$ if and only if
\textbf{OSC} is fulfilled. Then the problem is decidable whether \emph{int}$(E_{\mathbf{p},\mathbf{b}})\neq
\varnothing$.
\end{corollary}
\begin{remark}
For
$
E=\cup _{i=1}^{m}(\frac{1}{m}E+b_{i})
$
with $b_i\in \mathbb{Q}$  we have $\dim_S(E)=1$. See Example $3$ for different types.
\end{remark}

Lipschitz equivalence of fractals is also an interesting topic (\cite{CP},
\cite{FM}, \cite{DS}, \cite{RRW}, \cite{LLAD}, \cite{RWX}, \cite{X0}-\cite%
{XX1}). Let's consider the Lipschitz equivalence of the self-similar set
\begin{equation}  \label{Fi}
F=\cup _{i=1}^{m}(rF+b_{i})\text{\ \ with\ }r<1/m,
\end{equation}
where $r^{-1}>1$ is a P.V. number and $b_i\in \mathbb{Q}$.

The symbolic spaces are often regarded as an important tool for
researching Lipschitz equivalence between sets. We use $\Sigma_m^r$ to
denote the symbolic space $\{1,\cdots,m\}^\infty$ with the metric
\begin{equation*}
d(x_1x_2\cdots,y_1y_2\cdots)=r^{\min\{k|x_k\neq y_k\}}.
\end{equation*}
It is pointed out in \cite{XX1} that if $F$ fulfills OSC and totally
disconnectedness, then $F$ is Lipschitz equivalent to $\Sigma_m^r$, denoted
by $F\sim\Sigma_m^r$. Please refer to \cite{XX0} and \cite{Zhu} for weaker
versions of the above result. It is clear that
\begin{equation}\label{ss9}
s:=\dim_H(\Sigma_m^r)=\dim_S(F)=-\frac{\log m}{\log r}<1.
\end{equation}

\begin{corollary}\label{dengjia}
For $F$ defined as $(\ref{Fi}),$ then $F$ is Lipschitz equivalent to $\Sigma_m^r$
if and only if \textbf{OSC} is fulfilled. Then the problem is decidable
whether $F$ is Lipschitz equivalent to $\Sigma_m^r$.
\end{corollary}

\begin{proof}
Note that $\dim_H F\leq\dim_S F<1$, which implies that $F$ is totally
disconnected. Then by \cite{XX1}, it holds that \textbf{OSC}\ $%
\Longrightarrow$\ $F\sim\Sigma_m^r$.

By a usual discussion, we have $0<\mathcal{H}^s(\Sigma_m^r)<\infty, $ where $s$ is defined in (\ref{ss9}). If
suppose that $F\sim\Sigma_m^r$, then $0<\mathcal{H}^s(F)<\infty$. Hence, $F$
satisfies \textbf{OSC} by Schief's theorem.

Therefore, the corollary is obtained from the item (1) in Theorem \ref%
{Th:Main}.
\end{proof}

For example, let $\lambda =1/5,$  $\mathbf{p}=(1,1,1)$ and $\mathbf{b}=(0,7/10,4/5).$
Then Example 2 of \cite{WX} shows that

(1) $E_{\mathbf{p},\mathbf{b}}$ satisfies $\textbf{OSC};$

(2) $E_{\mathbf{p},\mathbf{b}}$ does not satisfy $\textbf{SSC}.$
\newline
By Corollary \ref{dengjia},  $E_{\mathbf{p},\mathbf{b}}\sim \Sigma _{3}^{1/5}.$

\subsection{Graph and algorithm}

$\ $

In fact, we can solve the above problems by constructing a directed graph $G$
and establishing the following criteria:

\begin{theorem}
\label{Th:OSC2}The \textbf{OSC} fails for $\{\lambda
^{p_{i}}x+b_{i}\}_{i=1}^{m}$ if and only if there is a finite path in $G$
starting from a point of $\Lambda $ and ending at $(0,0).$
\end{theorem}

\begin{theorem}
\label{Th:SSC2}The \textbf{SSC} fails for $\{\lambda
^{p_{i}}x+b_{i}\}_{i=1}^{m}$ if and only if there is an infinite path in $G$
starting from a point of $\Lambda .$
\end{theorem}

\begin{remark}
Using Dijkstra's algorithm, we can solve the existence of paths mentioned in
Theorems $2-3,$ see Lemma $\ref{Lem: QuestionAB}.$
\end{remark}

Now, we will describe the graph $G$ and $\Lambda $ which is a\ subset of
vertex set of $G.$

Let $B=\{0\}\cup \{b_{1},\cdots ,b_{m}\}.$ Denote $\mathcal{Z}-\mathcal{Z}%
=\{P:P$ is a polynomials with coefficients in $B-B\}.$ Set $T=\max_{i}p_{i}.$

Let $G$ be a directed finite graph with vertex set
\begin{equation*}
\Xi =\{n\in \mathbb{Z}:|n|\leq T-1\}\times \mathcal{C},
\end{equation*}%
with
\begin{equation}
\mathcal{C}=\{x:|x|\leq \frac{(1+\lambda ^{-T+1})\max_{i}|b_{i}|}{1-\lambda }%
,\ \lambda ^{-T+2}x=R(\lambda ^{-1})\text{\ with\ }R\in \mathcal{Z}-\mathcal{%
Z}\}.  \label{xi}
\end{equation}

For a finite word $\sigma =i_{1}i_{2}\cdots i_{k}$ with letters in $%
\{1,\cdots ,m\},$ then $S_{\sigma }=S_{i_{1}}\circ \cdots \circ S_{i_{k}}.$
We suppose
\begin{equation*}
S_{\sigma }x=\lambda ^{p_{\sigma }}x+b_{\sigma }\text{ with }p_{\sigma
}=p_{i_{1}}+\cdots +p_{i_{k}}.
\end{equation*}

Given $(p,b),$ $(p^{\prime },b^{\prime })\in \Xi ,$ there is a directed edge
$e$ from $(p,b)\ $to $(p^{\prime },b^{\prime }),$ if and only if there are
words $\alpha ,\beta $ with $1\leq p_{\alpha },p_{\beta }\leq 2T-1$ such that%
\begin{equation*}
p^{\prime }=p_{\alpha }+p-p_{\beta }\text{ and }b^{\prime }=\lambda
^{-p_{\beta }}b+\lambda ^{-p_{\beta }+p}b_{\alpha }-\lambda ^{-p_{\beta
}}b_{\beta }.
\end{equation*}%
Let $\Lambda =\Xi\cap \{(p_{j}-p_{i},\lambda ^{-p_{i}}b_{j}-\lambda
^{-p_{i}}b_{i}):i\neq j \}.$

\begin{remark}
In fact, for self-similar set of finite type, we have a scheme to test the $%
\mathbf{OSC}$ based on computing the exact Hausdorff dimension. In \cite{NW}%
, for the IFS $\{\phi _{j}(x)=\rho _{j}R_{j}x+b_{j},1\leq j\leq q\}$, Ngai
and Wang described a scheme for computing the Hausdorff dimension of its
attractors $K$:

\begin{itemize}
\item Step 1. select a suitable invariant open set $\Omega $;

\item Step 2. determine neighborhood type $\Omega (v)$;

\item Step 3. compute all neighborhood types;

\item Step 4. calculate the incidence matrix $S$ of neighborhood type;

\item Step 5. obtain the spectral radius $\lambda $ of $S.$
\end{itemize}
Then $\dim _{H}K=-\log \lambda /\log \rho ,$ where $\rho =\min_{j}\rho _{j}.$
By Theorem $1.2$ in \cite{NW} and Schief's theorem, we have
\begin{equation*}
\mathbf{OSC\ }\Longleftrightarrow \dim _{S}K=-\log \lambda /\log \rho .
\end{equation*}

However, from the view of algorithm, when implementing the scheme we have to
face some practical difficulties: $(1)$ an exhaustive search in the words set $%
\mathcal{V}_{k}$ $($defined as \cite{NW}$)$ is inevitable to look for all
neighborhood type $\Omega (v)$ of $v\in \mathcal{V}_{k}$ for every $(k\geq 0)
$; $(2)$ a heavy computation is inevitable to find all distinct neighborhood
types, due to the unknownness of the number of types.


For example, for a class of IFSs with the integral parameter $n>0$,
\begin{equation}
T_{1}=\frac{1}{3}x,\text{\ \ }T_{2}=\frac{1}{3}x+\frac{2}{3^{n+1}},\text{\ \
}T_{3}=\frac{1}{3}x+\frac{2}{3},  \label{Ti}
\end{equation}
Rao and Wen \cite{RW} obtained $2^{n}$ neighborhood types of $\{T_{i}\}$
which implies that an exponential calculated quantity is need to
differentiate neighborhood types of $\{T_{i}\}$, also see \cite{Guo} for
details. However, our algorithm in Section $4$ will avoid effectively these
difficulties. Example $1$ in Section $6$ tells us that a path with length $n$ in $G$ starting
from a point of $\Lambda $ and ending at $(0,0)$ can be found quickly in a
polynomial time, which leads to the \textbf{OSC}'s failing for $\{T_{i}\}$
in $(\ref{Ti})$ in virtue of Theorem $\ref{Th:OSC2}.$ 
%
\end{remark}

\medskip

The paper is organized as follows. Section 2 includes some preliminaries,
for example, Lemma \ref{Lem: QuestionAB} describes the decidability of
testing the existence of paths in Theorems \ref{Th:OSC2} and \ref{Th:SSC2}. In
Section 3, the \emph{recursive structure} is introduced to construct the
graph. In Section 4, we construct the graph and design these algorithms. In
Section \ref{sec:OSC}, we obtain a \emph{refinement algorithm} for the
\textbf{OSC}, especially we replace the upper bound $|x|\leq \frac{%
(1+\lambda ^{-T+1})\max_{i}|b_{i}|}{1-\lambda }$ in (\ref{xi}) by $|x|<\frac{%
2\max_{i}|b_{i}|}{(1-\lambda )^{2}}.$ In the last section, we will give some
examples to illustrate these algorithms.

\medskip

\section{Preliminaries}

\subsection{WSP and OSC}

$\ $\

We always fix $\mathbf{p}=(p_{1},\cdots ,p_{m})$ and $\mathbf{b}%
=(b_{1},\cdots ,b_{m})$ and write $E=E_{\mathbf{p},\mathbf{b}}$ for
notational convenience. Then $E=\cup _{i=1}^{m}S_{i}(E)$ where
\begin{equation*}
S_{i}(x)=\lambda ^{p_{i}}x+b_{i}
\end{equation*}%
with $p_{i}\in \mathbb{N}$ and $b_{i}\in \mathbb{Z}.$ Here $E\subset
\{x:|x|\leq \frac{\max_{i}|b_{i}|}{1-\lambda }\}.$ Write
\begin{equation}
T=\max_{i}p_{i}.  \label{T}
\end{equation}

Set $\Sigma =\{1,\cdots ,m\}^{\infty }$ and $\Sigma ^{\ast }=\cup
_{k=1}^{\infty }\{1,\cdots ,m\}^{k}.$ If $w=u\ast v$ for words $u\in \Sigma
^{\ast }$ and $w,v\in \Sigma ^{\ast }$ or $\Sigma ,$ we denote $u\prec w,$%
and write $w\backslash u=v.$ Let $[i]^{k}$ denote $\underset{k}{\underbrace{%
i\cdots i}}.$

For $\sigma \in \Sigma ^{\ast },$ we write
\begin{equation*}
S_{\sigma }(x)=\lambda ^{p_{\sigma }}x+b_{\sigma }.
\end{equation*}

Given $(p,b)\in \mathbb{Z}\times \mathbb{R}$, we let
\begin{equation*}
h_{(p,b)}(x)=\lambda ^{p}x+b.
\end{equation*}%
Then
\begin{eqnarray*}
h_{(p_{1},b_{1})}\circ h_{(p_{2},b_{2})} &=&h_{(p_{1}+p_{2},\lambda
^{p_{1}}b_{2}+b_{1})}, \\
h_{(p,b)}\circ h_{(-p,-\lambda ^{-p}b)} &=&h_{(-p,-\lambda ^{-p}b)}\circ
h_{(p,b)}=h_{(0,0)}.
\end{eqnarray*}%
Hence we obtain a group structure on $\mathbb{Z}\times \mathbb{R}$ satisfying%
\begin{equation*}
(p_{1},b_{1})\cdot (p_{2},b_{2})=(p_{1}+p_{2},\lambda ^{p_{1}}b_{2}+b_{1}),%
\text{ }(p,b)^{-1}=(-p,-\lambda ^{-p}b).
\end{equation*}%
We will omit $\cdot $ in the multiplication.

Let $\pi _{\mathbb{Z}}:\mathbb{Z}\times \mathbb{R\rightarrow Z}\ $and $\pi _{%
\mathbb{R}}:\mathbb{Z}\times \mathbb{R\rightarrow R}$ be the projections
defined by
\begin{equation*}
\pi _{\mathbb{Z}}(p,b)=p\text{ and }\pi _{\mathbb{R}}(p,b)=b.
\end{equation*}

For a letter $i\in \{1,\cdots ,m\},$ there is an element $(p_{i},b_{i})\in
\mathbb{N}\times \mathbb{R}$ such that the contracting similitude $%
S_{i}(x)=h_{(p_{i},b_{i})}.$ For a finite word $\sigma =i_{1}\cdots i_{k},$
then
\begin{equation*}
(p_{\sigma },b_{\sigma })=(p_{i_{1}},b_{i_{1}})\cdots (p_{i_{k}},b_{i_{k}})%
\text{ and }S_{\sigma }=h_{(p_{\sigma },b_{\sigma })}.
\end{equation*}%
Let $\mathcal{Z}=\{P:P$ is a polynomials with coefficients in $B\}.$

\begin{lemma}
\label{lem:pir}If $|p_{\sigma }-p_{\tau }|\leq T-1,$ then
\begin{equation*}
\pi _{\mathbb{R}}[(p_{\tau },b_{\tau })^{-1}(p_{\sigma },b_{\sigma
})]=\lambda ^{T-2}R(\lambda ^{-1}),
\end{equation*}%
where $R\in \mathcal{Z}-\mathcal{Z},$ i.e., $R$ is a polynomial with
coefficients in $B-B.$
\end{lemma}

\begin{proof}
In fact, for $\sigma =i_{1}\cdots i_{k},$ using induction, we can obtain that%
\begin{equation*}
S_{\sigma }(x)=\lambda ^{p_{\sigma }}x+Q(\lambda ),
\end{equation*}%
where $Q\in \mathcal{Z}$ with $\deg (Q)<p_{\sigma }.$ Then
\begin{equation*}
S_{\tau }^{-1}S_{\sigma }(x)=\lambda ^{p_{\sigma }-p_{\tau }}x+[\lambda^{-p_{\tau }}Q(\lambda )-P(\lambda ^{-1})]\text{ with }P\in \mathcal{Z}.
\end{equation*}%
Since $\deg (Q)\leq p_{\tau }+(T-2),$
\begin{equation*}
 \lambda^{-p_{\tau }}Q(\lambda )-P(\lambda ^{-1})=\lambda
^{T-2}R(\lambda ^{-1}),
\end{equation*}%
where $R(x)$ is a polynomial with coefficients in $B-B.$
\end{proof}

We recall the following result.

\begin{claim}
$($\cite{L}$)$ The \textbf{WSP} fulfills for $\{S_{i}\}_{i=1}^{m}.$
\end{claim}

In fact, using Lemma \ref{lem:pir},
\begin{equation*}
S_{\tau }^{-1}S_{\sigma }(x)=\lambda ^{p_{\sigma }-p_{\tau }}x+\pi _{\mathbb{%
R}}[(p_{\tau },b_{\tau })^{-1}(p_{\sigma },b_{\sigma })]
\end{equation*}%
with $\pi _{\mathbb{R}}[(p_{\tau },b_{\tau })^{-1}(p_{\sigma },b_{\sigma
})]=\lambda ^{T-2}R(\lambda ^{-1})$ if $|p_{\sigma }-p_{\tau }|\leq T-1.$ By
Garsia's result (\ref{garsia}), $\{R(\lambda ^{-1}):R\in \mathcal{Z}-%
\mathcal{Z}\}$ is discrete and $0\in \{R(\lambda ^{-1}):R\in \mathcal{Z}-%
\mathcal{Z}\}$. Hence we notice that if $S_{\tau }^{-1}S_{\sigma }$ is very
closed to the identity $\mathbf{id},$ then $p_{\sigma }=p_{\tau }$ and $\pi
_{\mathbb{R}}[(p_{\tau },b_{\tau })^{-1}(p_{\sigma },b_{\sigma })]=0,$ i.e.,
$S_{\tau }=S_{\sigma }.$ Therefore, the \textbf{WSP} fulfills. Furthermore
we have the following result.

\begin{lemma}
\label{lem:OSC}The \textbf{OSC} holds for $\{S_{i}\}_{i}$ if and only if
\begin{equation*}
S_{\sigma }\neq S_{\tau }\text{ for all distinct }\sigma ,\tau \in \Sigma
^{\ast }.
\end{equation*}
\end{lemma}

\begin{remark}
Proposition $1$ of \cite{Z} says that when $K=\cup _{i}f_{i}(K)\subset
\mathbb{R}^{n}$ does not lie in a hyperplane, then the \textbf{OSC }holds if and
only if the \textbf{WSP }fulfills and $f_{\mathbf{i}}\neq f_{\mathbf{j}}$ for
all $\mathbf{i}\neq \mathbf{j}.$ We can also verify Lemma $\ref{lem:OSC}$ by
using this result. In fact we only need to deal with the case that $E$ is a
singleton. Suppose $E=\{x_{0}\}$ is a singleton and the \textbf{OSC} fails.
Then $S_{i}(x_{0})=S_{j}(x_{0})=x_{0}$ for different letters $i,j.$ Let $%
\sigma =[i]^{p_{j}}$ and $\tau =[j]^{p_{i}},$ then $S_{\sigma
}(x-x_{0})=x_{0}+\lambda ^{p_{i}p_{j}}(x-x_{0})=S_{\tau }(x-x_{0})$ for all $%
x.$ Hence $S_{\sigma }=S_{\tau }.$
\end{remark}

Set
\begin{equation*}
\mathcal{C}=\{x:|x|\leq \frac{(1+\lambda ^{-T+1})\max_{i}|b_{i}|}{1-\lambda }%
\text{ and }\lambda ^{-T+2}x=R(\lambda ^{-1})\text{ with }R\in \mathcal{Z}-%
\mathcal{Z}\}.
\end{equation*}

\begin{lemma}
\label{lem:pv}$\mathcal{C}$ is a finite set with cardinality%
\begin{equation*}
\#\mathcal{C}\leq \frac{2^{l}(1+\lambda ^{-T+1})(\max_{i}|b_{i}|)^{l+1}}{%
\lambda ^{T-2}(1-\lambda )\prod\limits_{i=1}^{l}(1-|\eta _{i}|)},
\end{equation*}%
where $\eta _{1},\cdots ,\eta _{l}$ are algebraic conjugates of $\lambda
^{-1}\ $with $|\eta _{i}|<1$ for all $i.$
\end{lemma}

\begin{proof}
Notice that the height of $R$ is not greater than $2\max_{i}|b_{i}|.$ Using
Garsia's result (\ref{garsia}), for any $c,c^{\prime }\in \mathcal{C},$ we
have
\begin{equation*}
c=c^{\prime }\text{ or }|c-c^{\prime }|\geq \lambda ^{T-2}\delta
\end{equation*}%
where the constant%
\begin{equation*}
\delta =[2\max_{i}|b_{i}|]^{-l}\prod\limits_{i=1}^{l}(1-|\eta _{i}|).
\end{equation*}
\end{proof}

\subsection{Graph algorithm}

$\ $

Given a vertex $i_{0}$ in a nonnegative weighted and directed graph,
Dijkstra's algorithm \cite{D} finds the lowest cost path starting from $%
i_{0} $. Dijkstra's original algorithm does not use a min-priority queue and
runs in $O(|\mathcal{V}|^{2}),$ where $\mathcal{V}$ is the vertex set and $|%
\mathcal{V}|$ is the number of vertices. Furthermore, we can calculate all
the lowest costs $\{c(i,j)\}_{(i,j)\in \mathcal{V\times V}}$ with running
time $O(|\mathcal{V}|^{3}).$

In this paper, we will meet the following questions:

\textbf{Question (A):} Given a directed graph $G,$ for two different
vertexes $i_{0}$ and $j_{0}$ in $G,$ is there a directed path in $G$
starting from $i_{0}$ and ending at $j_{0}?$

\textbf{Question (B): }Given a directed graph $G$, for a vertex $i_{0}$ in $%
G,$ is there an \emph{infinity} directed path in $G$ starting from vertex $%
i_{0}?$

The following lemma is easy, but we give its proof here just to make this paper
self-contained.

\begin{lemma}
\label{Lem: QuestionAB}Questions $(A)$ and $(B)$ have polynomial time
algorithms with running time $O(|\mathcal{V}|^{2})$ and $O(|\mathcal{V}%
|^{3}) $ respectively.
\end{lemma}

\begin{proof}
Suppose $G=(\mathcal{V},\mathcal{E})$ with vertex set $\mathcal{V}$ and
directed edge set $\mathcal{E}$.

We will obtain construct a weighted graph $G^{\prime }$ with vertex set $%
\mathcal{V}$. For any ordered pair $(i,j)\in \mathcal{V\times V},$ we give
the following nonnegative weight
\begin{equation*}
w(i,j)=\left\{
\begin{array}{ll}
1 & \text{if there exists an edge in }\mathcal{E}\text{ from }i\text{ to }j,
\\
|\mathcal{V}|+1 & \text{otherwise.}%
\end{array}%
\right.
\end{equation*}%
Using Dijkstra's algorithm for $G^{\prime },$ we can calculate the lowest
cost $c(i_{0},j_{0})$ from $i_{0}$ to $j_{0}.$ Then there exists a directed
path in $G$ starting from $i_{0}$ and ending at $j_{0}$ if and only if the
lowest cost $c(i_{0},j_{0})\leq |\mathcal{V}|.$ This algorithm requires $O(|%
\mathcal{V}|^{2})$ time.

It is easy to find an equivalent question for Question (B):\textbf{\ }given
a directed graph $G$, for a vertex $i_{0}$ in $G,$ is there a directed path
in $G$ starting from vertex $i_{0}$ and ending at a vertex in a loop? To
solve this question, we can use the above $G^{\prime }$ again to calculate
all the lowest costs $\{c(i,j)\}_{(i,j)\in \mathcal{V\times V}}$ with
running time $O(|\mathcal{V}|^{3}).$ Considering all such points lying in
loops, i.e.,
\begin{equation*}
\Delta =\{i:c(i,i)\leq |\mathcal{V}|\},
\end{equation*}%
we find out that there is a path starting from vertex $i_{0}$ and ending at
a vertex in $\Delta $ if and only if $\min_{i\in \Delta }c(i_{0},i)\leq |%
\mathcal{V}|$.
\end{proof}

\medskip

\section{Recursive Structure with SSC and OSC}

Let
\begin{equation*}
\Xi =\{n\in \mathbb{Z}:|n|\leq T-1\}\times \mathcal{C}.
\end{equation*}

\subsection{Strong separation condition}

\begin{proposition}
\label{prop:ssc1}Suppose $i_{1}i_{2}\cdots $, \ $j_{1}j_{2}\cdots \in \Sigma
.$ If $S_{i_{1}i_{2}\cdots }(E)=S_{j_{1}j_{2}\cdots }(E)$, then there exist $%
\{\sigma _{k}\}_{k},\{\tau _{k}\}_{k}\subset \Sigma ^{\ast }\ $such that for
all $k\geq 1,$ $\sigma _{k}\prec i_{1}i_{2}\cdots $ and $\tau _{k}\prec $\ $%
j_{1}j_{2}\cdots $ satisfying $p_{\sigma _{k}},p_{\tau _{k}}\in ((k-1)T,kT],$%
\begin{equation*}
(p_{\tau _{k}},b_{\tau _{k}})^{-1}(p_{\sigma _{k}},b_{\sigma _{k}})\in \Xi .
\end{equation*}
\end{proposition}

\begin{proof}
For any $k\geq 1,$ we pick up the shortest prefix $\sigma _{k}\ $of $%
i_{1}i_{2}\cdots $ such that $p_{\sigma _{k}}\in (kT,(k+1)T]$ and pick up $%
\tau _{k}$ for $j_{1}j_{2}\cdots $ in the same way. Since
\begin{equation*}
\pi _{\mathbb{Z}}[(p_{\tau _{k}},b_{\tau _{k}})^{-1}(p_{\sigma
_{k}},b_{\sigma _{k}})]=p_{\sigma _{k}}-p_{\tau _{k}}
\end{equation*}
and $p_{\sigma _{k}},p_{\tau _{k}}\in (kT,(k+1)T],$ we have $\left\vert \pi
_{\mathbb{Z}}[(p_{\tau _{k}},b_{\tau _{k}})^{-1}(p_{\sigma _{k}},b_{\sigma
_{k}})]\right\vert \leq T-1.$

It follows from $S_{i_{1}i_{2}\cdots }(E)=S_{j_{1}j_{2}\cdots }(E)$ that
\begin{equation*}
S_{\sigma _{k}}(E)\cap S_{\tau _{k}}(E)\neq \varnothing ,
\end{equation*}%
i.e.,
\begin{equation*}
E\cap S_{\sigma _{k}}^{-1}S_{\tau _{k}}(E)\neq \varnothing ,
\end{equation*}%
where $S_{\sigma _{k}}^{-1}S_{\tau _{k}}(x)=\lambda ^{p_{\sigma
_{k}}-p_{\tau _{k}}}x+\pi _{\mathbb{R}}[(p_{\tau _{k}},b_{\tau
_{k}})^{-1}(p_{\sigma _{k}},b_{\sigma _{k}})].$ Suppose that
\begin{equation}  \label{wan}
x_{1}=\lambda ^{p_{\sigma _{k}}-p_{\tau _{k}}}(x_{2})+\pi _{\mathbb{R}%
}[(p_{\tau _{k}},b_{\tau _{k}})^{-1}(p_{\sigma _{k}},b_{\sigma _{k}})]
\end{equation}%
with $x_{1},x_{2}\in E.$ Using the fact $E\subset \{x:|x|\leq \frac{%
\max_{i}|b_{i}|}{1-\lambda }\},$ we have
\begin{eqnarray}  \label{wanli}
&&\left\vert \pi _{\mathbb{R}}[(p_{\tau _{k}},b_{\tau _{k}})^{-1}(p_{\sigma
_{k}},b_{\sigma _{k}})]\right\vert  \notag \\
&\leq &|x_{1}|+\lambda ^{p_{\sigma _{k}}-p_{\tau _{k}}}|x_{2}| \leq
(1+\lambda ^{p_{\sigma _{k}}-p_{\tau _{k}}})\frac{\max_{i}|b_{i}|}{1-\lambda
} \\
&\leq &(1+\lambda ^{-T+1})\frac{\max_{i}|b_{i}|}{1-\lambda }.  \notag
\end{eqnarray}

By Lemma \ref{lem:pir}, we obtain $\pi _{\mathbb{R}}[(p_{\tau _{k}},b_{\tau
_{k}})^{-1}(p_{\sigma _{k}},b_{\sigma _{k}})]=\lambda ^{T-2}R(\lambda
^{-1})\ $directly.
\end{proof}

\begin{proposition}
\label{prop:SSC2}Suppose $i_{1}i_{2}\cdots $, \ $j_{1}j_{2}\cdots \in \Sigma
.$ If for all $k\geq 1,$ there exist $\sigma _{k}\prec i_{1}i_{2}\cdots $
and $\tau _{k}\prec $\ $j_{1}j_{2}\cdots $ such that $p_{\sigma
_{k}},p_{\tau _{k}}\in ((k-1)T,kT]\ $and%
\begin{equation*}
\left\vert \pi _{\mathbb{R}}[(p_{\tau _{k}},b_{\tau _{k}})^{-1}(p_{\sigma
_{k}},b_{\sigma _{k}})]\right\vert \leq (1+\lambda ^{-T+1})\frac{%
\max_{i}|b_{i}|}{1-\lambda },
\end{equation*}%
then $S_{i_{1}i_{2}\cdots }(E)=S_{j_{1}j_{2}\cdots }(E).$
\end{proposition}

\begin{proof}
Considering the distance between $S_{\sigma _{k}}(0)$ and $S_{\tau _{k}}(0),$
we have $S_{\tau _{k}}^{-1}S_{\sigma _{k}}(x)=\lambda ^{p_{\sigma
_{k}}-p_{\tau _{k}}}x+\pi _{\mathbb{R}}[(p_{\tau _{k}},b_{\tau
_{k}})^{-1}(p_{\sigma _{k}},b_{\sigma _{k}})],$ i.e.,
\begin{equation*}
S_{\tau _{k}}^{-1}S_{\sigma _{k}}(0)=\pi _{\mathbb{R}}[(p_{\tau
_{k}},b_{\tau _{k}})^{-1}(p_{\sigma _{k}},b_{\sigma _{k}})],
\end{equation*}%
which implies%
\begin{eqnarray*}
&&|S_{\sigma _{k}}(0)-S_{\tau _{k}}(0)| \\
&=&\lambda ^{p_{\tau _{k}}}|S_{\tau _{k}}^{-1}S_{\sigma _{k}}(0)-0| \\
&\leq &\lambda ^{(k-1)T}(1+\lambda ^{-T+1})\frac{\max_{i}|b_{i}|}{1-\lambda }%
\rightarrow 0\text{ as }k\rightarrow \infty .
\end{eqnarray*}%
Letting $k\rightarrow \infty ,$ we obtain that
\begin{equation*}
S_{i_{1}i_{2}\cdots }(0)=S_{j_{1}j_{2}\cdots }(0)
\end{equation*}%
and thus $S_{i_{1}i_{2}\cdots }(E)=S_{j_{1}j_{2}\cdots }(E).$
\end{proof}

\subsection{Open set condition}

$\ $

According to Lemma \ref{lem:OSC}, we only need to check whether $S_{\sigma
}=S_{\tau }$ for some distinct $\sigma ,\tau \in \Sigma ^{\ast }$ or not?

\begin{proposition}
\label{Prop:OSC1}If $S_{\sigma }=S_{\tau },$ then there exists a positive
integer $M$ and
\begin{equation*}
\{\sigma _{k}\}_{k=1}^{M},\{\tau _{k}\}_{k=1}^{M}\subset \Sigma ^{\ast }\
\end{equation*}%
such that
\begin{equation*}
\sigma _{M}=\sigma ,\tau _{M}=\tau
\end{equation*}%
and for all $k\geq 1,$ $\sigma _{k}\prec \sigma $ and $\tau _{k}\prec \tau $
satisfying $p_{\sigma _{k}},p_{\tau _{k}}\in ((k-1)T,kT],$
\begin{equation*}
(p_{\tau _{k}},b_{\tau _{k}})^{-1}(p_{\sigma _{k}},b_{\sigma _{k}})\in \Xi .
\end{equation*}
\end{proposition}

\begin{proof}
Suppose that $p_{\sigma }=p_{\tau }\in ((M-1)T,MT].$ For $k<M,$ we can take
shortest prefix $\sigma _{k}$ (or $\tau _{k}$) of $\sigma $ (or $\tau $)
such that $p_{\sigma _{k}},p_{\tau _{k}}\in ((k-1)T,kT].$ Take $\sigma
_{M}=\sigma ,\tau _{M}=\tau .$ Let $\alpha _{0}=\sigma _{1}$ and $\beta
_{0}=\tau _{1}$ with $|\alpha _{0}|=|\beta _{0}|=1.$ Set $\alpha _{k}=\sigma
_{k+1}\backslash \sigma _{k}$ and $\beta _{k}=\tau _{k+1}\backslash \tau
_{k} $ for $k\geq 1.$ Then we also have
\begin{equation*}
1\leq p_{\alpha _{k}},p_{\beta _{k}}\leq 2T-1\text{ for all }k\geq 0.
\end{equation*}

Since $p_{\sigma _{k}},p_{\tau _{k}}\in ((k-1)T,kT],$ we have $|p_{\sigma
_{k}}-p_{\tau _{k}}|\leq T-1$ for all $k$, i.e.,
\begin{equation*}
\left\vert \pi _{\mathbb{Z}}[(p_{\tau _{k}},b_{\tau _{k}})^{-1}(p_{\sigma
_{k}},b_{\sigma _{k}})]\right\vert \leq T-1.
\end{equation*}%
As $S_{\sigma }(E)=S_{\tau }(E),$ then $S_{\sigma _{k}}(E)\cap S_{\tau
_{k}}(E)\neq \varnothing .$ As in the proof of Proposition \ref{prop:ssc1},
we have
\begin{equation*}
\left\vert \pi _{\mathbb{R}}[(p_{\tau _{k}},b_{\tau _{k}})^{-1}(p_{\sigma
_{k}},b_{\sigma _{k}})]\right\vert \leq (1+\lambda ^{-T+1})\frac{%
\max_{i}|b_{i}|}{1-\lambda }.
\end{equation*}%
By Lemma \ref{lem:pir}, we also have $\pi _{\mathbb{R}}[(p_{\tau
_{k}},b_{\tau _{k}})^{-1}(p_{\sigma _{k}},b_{\sigma _{k}})]\in \lambda
^{T-2}R(\lambda ^{-1}).$
\end{proof}

\subsection{Recursive structure}

$\ $\

\begin{definition}
Suppose $\{\alpha _{k}\}_{k=0}^{t},\{\beta _{k}\}_{k=0}^{t}\subset \Sigma
^{\ast }$ with $t\in \mathbb{N}$ or $t=\infty .$ We say that $%
\{(q_{k},c_{k})\}_{k=1}^{t+1}$ has the recursive structure w.r.t. $(\{\alpha
_{k}\}_{k=1}^{t},\{\beta _{k}\}_{k=0}^{t}),$ if
\begin{eqnarray*}
(q_{1},c_{1}) &=&(p_{\beta _{0}},b_{\beta _{0}})^{-1}(p_{\alpha
_{0}},b_{\alpha _{0}})\text{,} \\
(q_{k+1},c_{k+1}) &=&(p_{\beta _{k}},b_{\beta
_{k}})^{-1}(q_{k},c_{k})(p_{\alpha _{k}},b_{\alpha _{k}})\text{ for }1\leq
k\leq t.
\end{eqnarray*}
\end{definition}

By Propositions \ref{prop:ssc1}-\ref{prop:SSC2}, we have the following
result on the \textbf{SSC}.

\begin{proposition}
\label{Th:SSC1}The \textbf{SSC} fails for $\{\lambda
^{p_{i}}x+b_{i}\}_{i=1}^{m}$ if and only if there are
\begin{equation*}
\{\alpha _{k}\}_{k=0}^{\infty },\{\beta _{k}\}_{k=0}^{\infty }\subset \Sigma
^{\ast }
\end{equation*}%
with $1\leq p_{\alpha _{k}},p_{\beta _{k}}\leq 2T-1$ for all $k$ and $%
|\alpha _{0}|=|\beta _{0}|=1$, $\alpha _{0}\neq \beta _{0},$ such that
\begin{equation*}
(q_{k},c_{k})\in \Xi \text{ for all }k\geq 1,
\end{equation*}%
where $\{(q_{k},c_{k})\}_{k}$ has the recursive structure w.r.t. $(\{\alpha
_{k}\}_{k=0}^{\infty },\{\beta _{k}\}_{k=0}^{\infty })$.
\end{proposition}

\begin{proof}
If the strong separation condition fails, then there are $i_{1}i_{2}\cdots
,j_{1}j_{2}\cdots \in \Sigma $ with $i_{1}\neq j_{1}$ such that
\begin{equation*}
S_{i_{1}i_{2}\cdots }(E)=S_{j_{1}j_{2}\cdots }(E).
\end{equation*}%
Using the method in Proposition \ref{prop:ssc1}, we obtain that $\{\sigma
_{k}\}_{k},\{\tau _{k}\}_{k}$ with $\sigma _{1}=i_{1}=\alpha _{0}$ and $\tau
_{1}=j_{1}=\beta _{0}.$ Let $(q_{k},c_{k})=(p_{\tau _{k}},b_{\tau
_{k}})^{-1}(p_{\sigma _{k}},b_{\sigma _{k}}),$ and $\alpha _{k}=\sigma
_{k+1}\backslash \sigma _{k}$ and $\beta _{k}=\tau _{k+1}\backslash \tau
_{k}\ $where%
\begin{equation}
1\leq p_{\alpha _{k}},p_{\beta _{k}}\leq 2T-1.  \label{dd}
\end{equation}%
Since $S_{\tau _{k+1}}^{-1}S_{\sigma _{k+1}}=S_{\beta _{k}}^{-1}(S_{\tau
_{k}}^{-1}S_{\sigma _{k}})S_{\alpha _{k}},$ we have%
\begin{eqnarray*}
(q_{k+1},c_{k+1}) &=&(p_{\tau _{k+1}},b_{\tau _{k+1}})^{-1}(p_{\sigma
_{k+1}},b_{\sigma _{k+1}}) \\
&=&(p_{\beta _{k}},b_{\beta _{k}})^{-1}[(p_{\tau _{k}},b_{\tau
_{k}})^{-1}(p_{\sigma _{k}},b_{\sigma _{k}})](p_{\alpha _{k}},b_{\alpha
_{k}}) \\
&=&(p_{\beta _{k}},b_{\beta _{k}})^{-1}(q_{k},c_{k})(p_{\alpha
_{k}},b_{\alpha _{k}}).
\end{eqnarray*}%
Then by Proposition \ref{prop:ssc1}, for all $k\geq 1,$ we have
\begin{equation*}
(q_{k},c_{k})\in \Xi .
\end{equation*}

\medskip

On the other hand, if there are such $\{\alpha _{k}\}_{k\geq 0},\{\beta
_{k}\}_{k\geq 0}\subset \Sigma ^{\ast },$ let $\sigma _{k}=\alpha _{0}\ast
\alpha _{1}\ast \cdots \ast \alpha _{k-1}$ and $\tau _{k}=\beta _{0}\ast
\beta _{1}\ast \cdots \ast \beta _{k-1}.$ Letting $k\rightarrow \infty ,$ we
obtain $i_{1}i_{2}\cdots ,j_{1}j_{2}\cdots \in \Sigma $ such that $\sigma
_{k}\prec i_{1}i_{2}\cdots $ and $\tau _{k}\prec j_{1}j_{2}\cdots $ for all $%
k,$ and
\begin{equation*}
\left\vert \pi _{\mathbb{R}}[(p_{\tau _{k}},b_{\tau _{k}})^{-1}(p_{\sigma
_{k}},b_{\sigma _{k}})]\right\vert \leq (1+\lambda ^{-T+1})\frac{%
\max_{i}|b_{i}|}{1-\lambda }.
\end{equation*}%
By Proposition \ref{prop:SSC2}, we find that the strong separation condition
fails.
\end{proof}

By Proposition \ref{Prop:OSC1}, we have the following result on the \textbf{%
OSC}.

\begin{proposition}
\label{Th:OSC1}The \textbf{OSC} fails for $\{\lambda
^{p_{i}}x+b_{i}\}_{i=1}^{m}$ if and only if there exists a positive integer $%
M$ and
\begin{equation*}
\{\alpha _{k}\}_{k=0}^{M},\{\beta _{k}\}_{k=0}^{M}\subset \Sigma ^{\ast }
\end{equation*}%
with $1\leq p_{\alpha _{k}},p_{\beta _{k}}\leq 2T-1$ for all $k$ and $%
|\alpha _{0}|=|\beta _{0}|=1$, $\alpha _{0}\neq \beta _{0},$ such that
\begin{equation*}
(q_{M+1},c_{M+1})=(0,0)\text{ and }(q_{k},c_{k})\in \Xi \text{ for all }%
k\geq 1,
\end{equation*}%
where $\{(q_{k},c_{k})\}_{k=1}^{M+1}$ has the recursive structure w.r.t. $%
(\{\alpha _{k}\}_{k=0}^{M},\{\beta _{k}\}_{k=0}^{M})$.
\end{proposition}

\medskip

\section{\label{sec:alg}Directed Graph and Algorithm}

Consider a graph $G$ with vertex set
\begin{equation*}
\Xi =\{n\in \mathbb{Z}:|n|\leq T-1\}\times \mathcal{C}.
\end{equation*}

We will equip the graph with directed edge as follows. Given $(p,b),$ $%
(p^{\prime },b^{\prime })\in \Xi ,$ there is a directed edge $e$ from $%
(p,b)\ $to $(p^{\prime },b^{\prime }),$ if and only if there are words
\begin{equation*}
\alpha ,\beta \text{ with }1\leq p_{\alpha },p_{\beta }\leq 2T-1
\end{equation*}%
such that%
\begin{equation*}
(p^{\prime },b^{\prime })=(p_{\beta },b_{\beta })^{-1}(p,b)(p_{\alpha
},b_{\alpha }).
\end{equation*}

Notice that $(p_{i},b_{i})^{-1}(p_{j},b_{j})=(p_{j}-p_{i},\lambda
^{-p_{i}}b_{j}-\lambda ^{-p_{i}}b_{i}).$ Let
\begin{equation*}
\Lambda =\{(p_{i},b_{i})^{-1}(p_{j},b_{j}):i\neq j\text{ and }\lambda
^{-p_{i}}|b_{j}-b_{i}|\leq (1+\lambda ^{-T+1})\frac{\max_{i}|b_{i}|}{%
1-\lambda }\text{ }\}.
\end{equation*}%
We have $\Lambda \subset \Xi $ since $\lambda^{p_{i}}(b_{j}-b_{i})\in \mathcal{Z}-%
\mathcal{Z}.$

\medskip

Using Propositions \ref{Th:SSC1}-\ref{Th:OSC1}, we complete the proofs of
Theorems \ref{Th:OSC2}-\ref{Th:SSC2}. Furthermore, Theorem \ref{Th:Main}
follows from Theorems 2-3 and Lemma \ref{Lem: QuestionAB}.

\subsection{Realization of algorithm}

$\ $

\textbf{Step 1}: Set $(b_{1},\cdots ,b_{m})\in \mathbb{Z}^{m}$

If $(b_{1},\cdots ,b_{m})\in \mathbb{Q}^{m}\backslash \mathbb{Z}^{m},$ then
as in Remark \ref{remark:int}, we take $a\in N$ such that $b_{i}^{\prime
}=ab_{i}\in \mathbb{Z}$ for all $i,$ and thus we replace the IFS $\{\lambda
^{p_{i}}x+b_{i}\}_{i=1}^{m}$ by $\{\lambda ^{p_{i}}x+b_{i}^{\prime
}\}_{i=1}^{m}.$

\medskip

\textbf{Step 2}: Calculate the set $\mathcal{C}$

Let $y=\lambda ^{-T+2}x,$ then $\mathcal{C}=\lambda ^{T-2}\mathcal{D}$ where
\begin{equation*}
\mathcal{D}=\{y:|y|\leq d\text{ and }x=R(\lambda ^{-1})\text{ with }R\in
\mathcal{Z}-\mathcal{Z}\}.
\end{equation*}%
where
\begin{equation}
d=\lambda ^{-T+2}\frac{(1+\lambda ^{-T+1})\max_{i}|b_{i}|}{1-\lambda }.
\label{d}
\end{equation}%
Since $d\geq \frac{2\max_{i}|b_{i}|}{\lambda ^{-1}-1},$ we have
\begin{equation}
|y|>d\Longrightarrow |\lambda ^{-1}y+b|>d\text{ for all }b\in B-B.
\label{mm}
\end{equation}%
Let
\begin{equation}
\mathcal{D}_{k}=\{y\in \mathcal{D}:x=R(\lambda ^{-1})\text{\ with\ }\deg
(R)\leq k\}.
\end{equation}%
For any $R_{k}\in \mathcal{Z}-\mathcal{Z}$ with $\deg (R_{k})=k,$ there is a
polynomial $R_{k-1}\in \mathcal{Z}-\mathcal{Z}$ with $\deg (R_{k-1})=k-1$
and $b\in B-B$ such that
\begin{equation*}
R_{k}(\lambda ^{-1})=\lambda ^{-1}R_{k-1}(\lambda ^{-1})+b.
\end{equation*}%
It follows from (\ref{mm}) that $R_{k-1}(\lambda ^{-1})\in \mathcal{D}_{k-1}$
if $R_{k}(\lambda ^{-1})\in \mathcal{D}_{k}.$ Therefore,
\begin{equation}
\mathcal{D}_{k}=\{y:|y|\leq d\}\cap (\lambda ^{-1}\mathcal{D}_{k-1}+(B-B)).
\label{indu}
\end{equation}

We have the following recursive algorithm: \newline
(1) Let $\mathcal{D}_{0}=\{y:|y|\leq d\}\cap (B-B);$ \newline
(2) For $k\geq 1$, let%
\begin{equation*}
\mathcal{D}_{k}=\{y:|y|\leq d\}\cap (\lambda ^{-1}\mathcal{D}_{k-1}+(B-B)),
\end{equation*}%
or
\begin{equation}  \label{dk}
\mathcal{D}_{k}=\{y:|y|\leq d\}\cap (\lambda ^{-1}(\mathcal{D}%
_{k-1}\backslash \mathcal{D}_{k-2})+(B-B));
\end{equation}%
(3) If we find a smallest $k_{0}$ such that $\mathcal{D}_{k_{0}+1}=\mathcal{D%
}_{k_{0}},$ then let $\mathcal{D}=\mathcal{D}_{k_{0}};$ \newline
(4) Let $\mathcal{C}=\lambda ^{T-2}\mathcal{D}.$

\medskip

\textbf{Step 3}: Draw all\ edges in graph

We obtain $\mathcal{C}$\ in step 2, then we have the vertex set
\begin{equation}  \label{xixi}
\Xi =([-T+1,T-1]\cap \mathbb{Z})\times \mathcal{C}.
\end{equation}%
We draw a directed edge $e$ from $(p,b)\in \Xi \ $to $(p^{\prime },b^{\prime
})\in \Xi ,$ if there are words $\alpha ,\beta $ with $1\leq p_{\alpha
},p_{\beta }\leq 2T-1$ such that%
\begin{equation*}
(p^{\prime },b^{\prime })=(p_{\alpha }+p-p_{\beta },\lambda ^{-p_{\beta
}+p}b_{\alpha }+\lambda ^{-p_{\beta }}b-\lambda ^{-p_{\beta }}b_{\beta }).
\end{equation*}%
Let
\begin{equation}  \label{da-lambda}
\Lambda =\Xi \cap \{(p_{j}-p_{i},\lambda ^{-p_{i}}b_{j}-\lambda
^{-p_{i}}b_{i}):i\neq j\}.
\end{equation}

\medskip

\textbf{Step 4}: Test the existence of corresponding paths

Using the methods stated in Lemma \ref{Lem: QuestionAB}, we can test the
existence of paths.

\medskip

Then we solve the problems, since the \textbf{SSC} fails if and only if we
find an infinite path starting from a point of $\Lambda ,$ and the \textbf{%
OSC} fails if and only if there is a finite path starting at a point of $%
\Lambda $ and ending at $(0,0).$

\subsection{\label{subsec:T=1}The case that $T=1$}

$\ $

Suppose $T=1,$ i.e., $p_{1}=\cdots =p_{m}=1.$ Then
\begin{equation*}
\digamma =\{\sigma :1\leq p_{\sigma }\leq 2T-1\}=\{1,\cdots ,m\}.
\end{equation*}%
Since $\Lambda =\Xi \cap \{(0,\lambda ^{-1}(b_{j}-b_{i}):i\neq j\}$ and
\begin{equation*}
\pi _{\mathbb{Z}}((1,b_{i})^{-1}(0,c)(1,b_{j}))=0,
\end{equation*}%
we conclude that any vertex $(q,c)$ in a directed path starting from a point
of $\Lambda $ must has
\begin{equation}
q=0.  \label{q=0}
\end{equation}

Since $p_{\sigma _{k}}=p_{\tau _{k}}=k$ in (\ref{wan}), as in (\ref{wanli})
we have
\begin{equation*}
\left\vert \pi _{\mathbb{R}}[(p_{\tau _{k}},b_{\tau _{k}})^{-1}(p_{\sigma
_{k}},b_{\sigma _{k}})]\right\vert =|x_{1}-x_{2}|\leq \frac{\max_{b\in
B-B}|b|}{1-\lambda },
\end{equation*}%
where $x_{1},x_{2}\in E.$ Then we can restrict the vertices in
\begin{equation}
\{0\}\times \mathcal{C}^{\prime },  \label{0*C}
\end{equation}%
where
\begin{equation}\label{cprime}
\mathcal{C}^{\prime }=\{x:|x|\leq \frac{\max_{b\in B-B}|b|}{1-\lambda
} \text{\ and\  } x=\lambda ^{-1}R(\lambda ^{-1}) \text{\ with\ } R\in \mathcal{Z}-\mathcal{Z}%
\}.
\end{equation}%

\subsection{Running time of algorithms}
$\ $

Using Lemma \ref{lem:pv} and $(1+\lambda ^{-T+1})\leq 2\lambda ^{-T+1}$, we
have
\begin{equation}  \label{jingc}
\#\mathcal{C}\leq \frac{2^{l}(1+\lambda ^{-T+1})(\max_{i}|b_{i}|)^{l+1}}{%
\lambda ^{T-2}(1-\lambda )\prod\limits_{i=1}^{l}(1-|\eta _{i}|)}\leq
a_{\lambda }b_{T}c_{B},
\end{equation}%
where
\begin{equation*}
a_{\lambda }=2^{l+1}(1-\lambda )^{-1}\prod\limits_{i=1}^{l}(1-|\eta
_{i}|)^{-1},\text{ }b_{T}=\lambda ^{-2T+3}\text{ and }c_{B}=(%
\max_{i}|b_{i}|)^{l+1}.
\end{equation*}
Notice that%
\begin{equation*}
\#\Xi =(2T-1)\#\mathcal{C}\leq a_{\lambda }[(2T-1)b_{T}]c_{B}.
\end{equation*}

(1) Time for calculating $\mathcal{C}$:

According to Step 3, we need running time
\begin{equation*}
t_{1}\leq \sum\nolimits_{i=0}^{k_{0}}\#(B-B)\cdot \#(\mathcal{D}%
_{k}\backslash \mathcal{D}_{k-1})\leq \#(B-B)\cdot \#\mathcal{C}
\end{equation*}%
i.e.,
\begin{equation*}
t_{1}\leq (m+1)^{2}\cdot \#\mathcal{C}.
\end{equation*}

(2) Time for finding all edges:

According to Step 4, we need to calculate
\begin{equation*}
(p_{\alpha }+p-p_{\beta },\lambda ^{-p_{\beta }+p}b_{\alpha }+\lambda
^{-p_{\beta }}b-\lambda ^{-p_{\beta }}b_{\beta })
\end{equation*}%
for all $(p,b)\in \Xi $ and words $\alpha ,\beta \in \digamma =\{\sigma
:1\leq p_{\sigma }\leq 2T-1\}$ with $\#\digamma \leq m+m^{2}+\cdots
+m^{2T-1}<m^{2T}.$ Then running time
\begin{equation*}
t_{2}\leq \#\Xi \cdot (\#\digamma )^{2}<m^{4T}(\#\Xi ).
\end{equation*}%
If $(\#\digamma )$ is large, we need more time to calculate.

(3) Time for finding corresponding paths:

By Lemma \ref{Lem: QuestionAB}, when the graph has been constructed, using
Dijkstra's algorithm we can test the non-existences of the corresponding
paths w.r.t. the \textbf{OSC} and the \textbf{SSC} with running time
\begin{equation*}
t_{3}\leq C(\#\Xi )^{2}\text{ and }t_{3}^{\prime }\leq C(\#\Xi )^{3}\text{
respectively,}
\end{equation*}%
where $C$ is the constant related to the original Dijkstra's algorithm.

Therefore we have
\begin{eqnarray*}
t_{1}+t_{2}+t_{3} &\leq &(m+1)^{2}\#\mathcal{C+}m^{4T}(\#\Xi )+C(\#\Xi )^{2},
\\
t_{1}+t_{2}+t_{3}^{\prime } &\leq &(m+1)^{2}\#\mathcal{C+}m^{4T}(\#\Xi
)+C(\#\Xi )^{3},
\end{eqnarray*}%
where $(\#\Xi )=(2T-1)\#\mathcal{C}$ and $\#\mathcal{C}\leq a_{\lambda
}b_{T}c_{B}.$

\begin{remark}
If $T=1$, $\#\mathcal{C}$ in $(\ref{jingc})$ can be replaced by $\#\mathcal{C}%
^{\prime }$ in $(\ref{0*C}),$ and
\begin{equation*}
\#\mathcal{C}^{\prime}\leq 2a_\lambda c^{\prime }_B+1
\end{equation*}
where $a_\lambda=(1-\lambda)^{-1}$ and $c^{\prime }_B=\max_{b\in {B-B}}|b|$.
\end{remark}

\begin{remark}
When $\lambda^{-1}$ is an integer, by $(\ref{d})$ and $(\ref{dk}),$ the
cardinality of $\mathcal{C}$ satisfies that
\begin{equation*}
\#\mathcal{C}=\#\mathcal{D}_{k_{0}}\leq2d+1=2b_Tc_B+1
\end{equation*}
where $b_T=(\lambda^{2-T}+\lambda^{3-2T})(1-\lambda)^{-1}$ and $%
c_{B}=\max_{i}|b_{i}|$.
\end{remark}

\medskip

\section{Refinement Algorithm on OSC}

\label{sec:OSC}

\subsection{Recursive structure}

$\ $

Suppose
\begin{equation*}
S_{\sigma }=S_{\tau }
\end{equation*}%
with $\sigma =i_{1}\cdots \in \Sigma ^{\ast }$ and $\tau =j_{1}\cdots \in
\Sigma ^{\ast }.$ Without loss of generality, we assume that $p_{i_{1}}\geq
p_{j_{1}}.$

We will use $\sigma _{i},\alpha _{i},\beta _{i}$ with different meaning as
above, but they have the same structure. We say that two numbers $a,b$ have
different signs if $ab\leq 0.$

Let $\sigma _{1}=\alpha _{0}=i_{1}$ and $\tau _{1}=\beta _{0}=j_{1}.$ Take $%
\beta _{1}$ the shortest prefix of $\tau \backslash \beta _{0}$ such that
\begin{equation*}
p_{\sigma _{1}}-p_{\tau _{1}},p_{\sigma _{1}}-p_{\tau _{1}\ast \beta _{1}}
\end{equation*}%
have different signs. Let $\tau _{2}=\tau _{1}\ast \beta _{1}.$ Take $\alpha
_{1}$ the shortest prefix of $\sigma \backslash \alpha _{0}$ such that
\begin{equation*}
p_{\sigma _{1}}-p_{\tau _{2}},p_{\sigma _{1}\ast \alpha _{1}}-p_{\tau _{2}}
\end{equation*}%
have different signs and let $\sigma _{2}=\sigma _{1}\ast \alpha _{1}.$
Inductively, we have
\begin{equation*}
p_{\sigma _{1}}-p_{\tau _{1}},p_{\sigma _{1}}-p_{\tau _{2}},p_{\sigma
_{2}}-p_{\tau _{2}},\cdots ,p_{\sigma _{k-1}}-p_{\tau _{k}},p_{\sigma
_{k}}-p_{\tau _{k}},\cdots
\end{equation*}%
have alternative signs, where
\begin{equation}
p_{\sigma _{k}}-p_{\tau _{k}}\geq 0.  \label{positive}
\end{equation}

There exists an integer $M$ such that $\sigma _{M}=\sigma .$ We will
distinguish two following cases $\tau _{M}=\tau $ or $\tau _{M}\neq \tau .$
For the latter, we will replace $\tau _{M}$ by $\tau .$ Since $p_{\sigma
}=p_{\tau },$
\begin{equation*}
p_{\sigma _{1}}-p_{\tau _{1}},\cdots ,p_{\sigma _{k-1}}-p_{\tau
_{k-1}},p_{\sigma _{k}}-p_{\tau _{k}},\cdots ,p_{\sigma _{M-1}}-p_{\tau
_{M}},p_{\sigma }-p_{\tau }
\end{equation*}%
have alternative signs. Notice that
\begin{equation*}
\sigma =\alpha _{0}\cdots \alpha _{M}\text{, }\sigma _{k}=\alpha _{0}\cdots
\alpha _{k-1}\text{ and }\tau =\beta _{0}\cdots \beta _{M}\text{, }\tau
_{k}=\beta _{0}\cdots \beta _{k-1},
\end{equation*}%
and
\begin{equation}
1\leq p_{\alpha _{k}},p_{\beta _{k}}\leq 2T-1\ \text{and }0\leq p_{\sigma
_{k}}-p_{\tau _{k}}\leq T-1.  \label{gg}
\end{equation}

Suppose $\{(q_{k},c_{k})\}_{k}$ has recursive structure w.r.t. $(\{\alpha
_{k}\}_{k=0}^{M},\{\beta _{k}\}_{k=0}^{M})$. Then%
\begin{equation*}
0\leq q_{k}\leq T-1\text{ and }c_{k}=\lambda ^{T-2}R(\lambda ^{-1})\text{
with }R\in \mathcal{Z}-\mathcal{Z}
\end{equation*}%
due to Lemma \ref{lem:pir} and (\ref{gg})$.$ Notice that $%
(q_{k+1},c_{k+1})=(p_{\beta _{k}},b_{\beta
_{k}})^{-1}(q_{k},c_{k})(p_{\alpha _{k}},b_{\alpha _{k}})$ where
\begin{equation*}
q_{k},q_{k}-p_{\beta _{k}},q_{k+1}(=q_{k}-p_{\beta _{k}}+p_{\alpha _{k}})
\end{equation*}%
have alternative signs. Let
\begin{equation*}
d^{\ast }=\frac{2}{1-\lambda }\left( \frac{\max_{i}|b_{i}|}{1-\lambda }%
\right) .
\end{equation*}

\begin{lemma}
Suppose $(q^{\prime },c^{\prime })=(p_{\beta },b_{\beta
})^{-1}(q,c)(p_{\alpha },b_{\alpha })$ with $q\geq 0.$ If $|c|\geq d^{\ast
}, $ then $|c^{\prime }|\geq |c|.$
\end{lemma}

\begin{proof}
Suppose $a>1$ and $|b|\leq t.$ If $|x|\geq \frac{t}{a-1}$, then
\begin{equation}
|ax+b|\geq a|x|-t\geq |x|.  \label{aa}
\end{equation}

Note that
\begin{equation*}
c^{\prime }=\lambda ^{-p_{\beta }}c+\lambda ^{-p_{\beta }}(\lambda
^{q}b_{\alpha }-b_{\beta }).
\end{equation*}%
with $|b_{\alpha }|,|b_{\beta }|\leq \frac{\max_{i}|b_{i}|}{1-\lambda }.$
Applying $a=\lambda ^{-p_{\beta }}$, $b=\lambda ^{-p_{\beta }}(\lambda
^{q}b_{\alpha }-b_{\beta })$ and $t=2\lambda ^{-p_{\beta }}\frac{%
\max_{i}|b_{i}|}{1-\lambda },$ we notice that
\begin{equation*}
\text{if }|c|\geq \frac{2\lambda ^{-p_{\beta }}}{\lambda ^{-p_{\beta }}-1}%
\left( \frac{\max_{i}|b_{i}|}{1-\lambda }\right) ,\text{ then }|c^{\prime
}|\geq |c|.
\end{equation*}%
Since the function $\frac{2\lambda ^{-x}}{\lambda ^{-x}-1}$ is decreasing
and $p_{\beta }\geq 1,$ we take $d^{\ast }$ for $p_{\beta }=1$ and complete
the proof.
\end{proof}

Notice that $\mathbf{id}=h_{(0,0)},$ we have the following corollary.

\begin{corollary}
Suppose $S_{\sigma }=S_{\tau }$ and $\{(q_{k},c_{k})\}_{k}$ has recursive
structure w.r.t. $(\{\alpha _{k}\}_{k=0}^{M},\{\beta _{k}\}_{k=0}^{M})$ as
above. Then $|c_{k}|<d^{\ast }$ for all $k.$
\end{corollary}

\subsection{Graph and criterion}

$\ $

Let
\begin{equation*}
\Theta =\{n\in \mathbb{Z}:0\leq n\leq T-1\}\times \mathcal{A}
\end{equation*}%
where
\begin{equation*}
\mathcal{A}=\{x:|x|<\frac{2\max_{i}|b_{i}|}{(1-\lambda )^{2}},\text{ }%
\lambda ^{-T+2}x=R(\lambda ^{-1})\text{ with }R\in \mathcal{Z}-\mathcal{Z}\}.
\end{equation*}

\begin{remark}
When $T\geq 3$, $(1+\lambda ^{-T+1})\frac{\max_{i}|b_{i}|}{1-\lambda }\geq
\frac{2\max_{i}|b_{i}|}{(1-\lambda )^{2}}.$ When $T\ $is large enough,
\begin{equation*}
(1+\lambda ^{-T+1})\frac{\max_{i}|b_{i}|}{1-\lambda }\gg \frac{%
2\max_{i}|b_{i}|}{(1-\lambda )^{2}}.
\end{equation*}%
Then $\#\mathcal{A}$ seems to be much less than $\#\mathcal{C}$.
\end{remark}

Given two vertices $(q,c)$, $(q^{\prime },c^{\prime })\in \Theta ,$ we draw
a directed edge from $(q,c)$ and $(q^{\prime },c^{\prime }),$ if and only if
there are words $\alpha ,\beta $ with $1\leq p_{\alpha _{k}},p_{\beta
_{k}}\leq 2T-1$ such that
\begin{equation*}
(q^{\prime },c^{\prime })=(p_{\beta },b_{\beta })^{-1}(q,c)(p_{\alpha
},b_{\alpha })
\end{equation*}
and $q,q-p_{\beta },q^{\prime }$ have alternative signs.

Then we obtain a graph $G^{\ast }$ with vertex set $\Theta $ and edge set
defined above$.$ Let
\begin{equation*}
\Psi =\{(p_{i},b_{i})^{-1}(p_{j},b_{j}):p_{j}\geq p_{i},\text{ }i\neq j\text{
and }\lambda ^{-p_{i}}|b_{j}-b_{i}|<\frac{2\max_{i}|b_{i}|}{(1-\lambda )^{2}}%
\text{ }\}.
\end{equation*}%
Then $\Psi \subset \Theta .$ We have the following criterion.

\begin{theorem}
The \textbf{OSC} fails if and only if there is a directed path in $G^{\ast }$
starting at a point of $\Psi $ and ending at $(0,0).$
\end{theorem}

\medskip

\section{Examples}

\begin{example}
Let $\lambda =\frac{1}{3}.$ For given $n\in \mathbb{N}$, we consider the
self-similar set $E_{\mathbf{p},\mathbf{b}}$ with
\begin{equation*}
\mathbf{p}=(1,1,1),\text{\ }\mathbf{b}=(0,\frac{2}{3^{n+1}},\frac{2}{3}).
\end{equation*}%
Notice that $\lambda ^{-1}=3$ is a P.V. number. By Remark $\ref{R:a},$ take $a=3^{n+1}$ and replace $(0,%
\frac{2}{3^{n+1}},\frac{2}{3})$ by $\mathbf{b}=(0,2,2\cdot3^n).$ Then
\begin{equation*}
B=\{0,2,2\cdot3^n\}
\end{equation*}
and
\begin{equation*}
B-B=\{-2\cdot3^n,2(1-3^n),-2,0,2,2(3^n-1),2\cdot3^n\}.
\end{equation*}

Since $T=1,$ by the discussion of Subsection $\ref{subsec:T=1},$ we have
\begin{equation*}
\mathcal{C}^{\prime}=\{m: m \text{ \ is \ even \ and \ } m\in [-3^{n+1},3^{n+1}]\cap \mathbb{Z}\}
\end{equation*}
and
$\mathcal{C}^{\prime} \cap \{\lambda ^{-1}(b_{i}-b_{j})\}_{i\neq j}=\{-6,6\},$ i.e.,
\begin{equation*}
\Lambda =\{(0,-6),(0,6)\}.
\end{equation*}
Considering the vertices $\{0\}\times \mathcal{C}^{\prime}$ and using step $3$
and step $4$ of algorithm, we can obtain a path
\begin{equation*}
(0,6)\rightarrow
(0,2\cdot3^2)\rightarrow(0,2\cdot3^3)\rightarrow\cdots\rightarrow(0,2%
\cdot3^n)\rightarrow(0,0)\rightarrow(0,0)\rightarrow(0,0)\cdots
\end{equation*}
 starting from $(0,6)$.
Then neither \textbf{OSC}
nor \textbf{SSC} is fulfilled. And thus $\dim_H E_{\mathbf{p},\mathbf{b}}\neq\dim_S E_{\mathbf{p},\mathbf{b}}=1$ and \emph{int}$(E_{\mathbf{p},\mathbf{b}})=\varnothing$
due to Corollarys $\ref{weishu}$ and $\ref{int}.$

\bigskip
\end{example}

\begin{example}
Let $\lambda =\frac{1}{3}.$ We consider the self-similar set $E_{\mathbf{p},%
\mathbf{b}}$ with
\begin{equation*}
\mathbf{p}=(1,2,1),\text{\ }\mathbf{b}=(0,\frac{11}{18},\frac{2}{3}).
\end{equation*}%
Take $a=18\ $and replace $(0,\frac{11}{18},\frac{2}{3})$ by $\mathbf{b}%
=(0,11,12).$ Then
\begin{equation*}
B=\{0,11,12\}
\end{equation*}
and
\begin{equation*}
B-B=\{-12,-11,-1,0,1,11,12\}.
\end{equation*}

Notice that $T=2$. Then we have, by (\ref{d}),
\begin{equation*}
d=\lambda^{-T+2}\frac{(1+\lambda^{-T+1})\max_{b\in B}|b|}{1-\lambda }=72,
\end{equation*}%
Using step $2$ of algorithm, we have
\begin{equation*}
\mathcal{C}=[-72,72]\cap \mathbb{Z},
\end{equation*}
and
\begin{equation*}
\Xi=\{-1,0,1\}\times\mathcal{C}.
\end{equation*}
By $(\ref{da-lambda})$ we also have
\begin{equation*}
\Lambda =\{(-1,9),(0,-36),(0,36),(1,-3),(1,33)\}.
\end{equation*}
Using step $3$ and step $4$ of algorithm, we find that there exists no a
directed path starting from any element in $\Lambda$ and ending at $(0,0)$.
Hence, $E_{\mathbf{p}, \mathbf{b}}$ satisfies \textbf{OSC}.

However, we can
obtain an infinite directed path
\begin{equation*}
(-1,9)\rightarrow
(-1,-9)\rightarrow(-1,-27)\rightarrow(0,18)\rightarrow(0,18)%
\rightarrow(0,18)\cdots
\end{equation*}
starting from $(-1,9)\in\Lambda.$  Then \textbf{SSC} fails.

\bigskip
\end{example}

\begin{example}
Let $\lambda =\sqrt2-1.$ Then $\lambda^{-1} =\sqrt2+1$ is a P.V. number
called the silver ratio. We consider the self-similar set $E_{\mathbf{p},%
\mathbf{b}}$ with
\begin{equation*}
\mathbf{p}=(1,2,1),\text{\ }\mathbf{b}=(0,\frac{2}{5},\frac{1}{2}).
\end{equation*}%
Take $a=10\ $and replace $(0,\frac{2}{5},\frac{1}{2})$ by $\mathbf{b}%
=(0,4,5).$ Then
$$
B=\{0,4,5\}
$$
and
$
B-B=\{-5,-4,-1,0,1,4,5\}.
$

Notice that $T=2$. Then by $(\ref{d})$ we have
\begin{equation*}
d=\lambda^{-T+2}\frac{(1+\lambda^{-T+1})\max_{b\in B}|b|}{1-\lambda }=15+10%
\sqrt{2},
\end{equation*}%
And we can obtain, using step $2$ of algorithm, the set $\Xi$ containing $ 1059$
elements which will not be listed here. By $(\ref{da-lambda})$ we also have
\begin{eqnarray*}
\Lambda &=& \{(1,4+4\sqrt2),(0,5+5\sqrt2),(-1,-12-8\sqrt2), \\
\ &\ & (-1,3+2\sqrt2),(0,-5-5\sqrt2),(1,-1-\sqrt2)\}.
\end{eqnarray*}

In virtue of Dijkstra's method, using step $3$ and step $4$ of
algorithm, we can obtain a path
\begin{equation*}
(-1,3+2\sqrt2)\rightarrow
(-1,-9-6\sqrt2)\rightarrow(-1,-6-5\sqrt2)\rightarrow(0,0)\rightarrow(0,0)%
\rightarrow(0,0)\cdots,
\end{equation*}
starting from $%
(-1,3+2\sqrt2)\in\Lambda,$ which implies that neither \textbf{OSC} nor \textbf{SSC} is fulfilled. Using $\dim_S (E_{\mathbf{p},\mathbf{b}})=1,$
we have \emph{int}$(E_{\mathbf{p},\mathbf{b}})=\varnothing$ due to Corollary $\ref{int}.$

\bigskip
\end{example}

%
%
%

\end{document}